\documentclass{amsart}

\usepackage{hyperref}

\usepackage{amssymb}
\usepackage{amsmath}
\usepackage{amsthm}
\usepackage{enumerate}
\usepackage{graphicx}
\usepackage{color}
\usepackage{centernot}

\theoremstyle{plain}

\newtheorem{theorem}{Theorem}[section]

\newtheorem{proposition}[theorem]{Proposition}

\newtheorem{corollary}[theorem]{Corollary}
\newtheorem{question}[theorem]{Question}

\theoremstyle{definition}
\newtheorem{definition}[theorem]{Definition}

\theoremstyle{remark}
\newtheorem*{remark}{Remark}

\newcommand{\mC}{\mathbb{C}}
\newcommand{\mR}{\mathbb{R}}

\newcommand{\mD}{\mathbb{D}}

\newcommand{\hC}{\widehat{\mathbb{C}}}

\newcommand{\RiemannSphere}{\widehat{\mathbb{C}}}

\begin{document}

\title{Continuous analytic capacity and holomorphic motions}

\date{November 6, 2024}

\author[M. Younsi]{Malik Younsi}
\address{Department of Mathematics, University of Hawaii Manoa, Honolulu, HI 96822, USA.}
\email{malik.younsi@gmail.com}



\begin{abstract}
We construct a compact set whose continuous analytic capacity does not vary continuously under a certain holomorphic motion, thereby answering a question of Paul Gauthier. Our example is inspired by holomorphic dynamics and relies on the works of Bishop--Carleson--Garnett--Jones and Browder--Wermer relating tangent points of Jordan curves, harmonic measure and Dirichlet algebras. Our approach also provides a new proof of a result of Ransford, Younsi and Ai on the variation of analytic capacity under holomorphic motions. In addition, we show that extremal functions for continuous analytic capacity may not exist.
\end{abstract}

\maketitle

\section{Introduction and statement of main results}
\label{sec1}

Let $E$ be a compact subset of the complex plane $\mC$ and let $D:=\hC \setminus E$ denote the complement of $E$ in the Riemann sphere $\hC$, so that $\infty \in D$. The \textit{analytic capacity} of $E$ is defined by
$$\gamma(E):= \sup \{|f'(\infty)|: f \in H^\infty(D), |f| \leq 1 \,\, \mbox{on} \,\, D\}.$$
Here $H^\infty(D)$ denotes the space of all bounded analytic functions on $D$ and $f'(\infty)$ is defined by
$$f'(\infty):= \lim_{z \to \infty} z(f(z)-f(\infty)).$$
Analytic capacity was introduced by Ahlfors in \cite{AHL} for the study of a problem of Painlev\'e from 1888 asking for a geometric characterization of the compact sets $E \subset \mathbb{C}$ that are \textit{removable} for bounded analytic functions, in the sense that all functions in $H^\infty(D)$ are constant. It is easily seen that such removable sets coincide precisely with the sets of zero analytic capacity. Despite recent advances, analytic capacity is notoriously hard to estimate in general and some of its properties remain quite mysterious. For instance, it is not known whether analytic capacity is subadditive, i.e.,
$$\gamma(E \cup F) \leq \gamma(E) + \gamma(F)$$
for all compact sets $E,F \subset \mC$. For more information on analytic capacity, we refer the reader to \cite{DUD}, \cite{GAR}, \cite{TOL}, \cite{YOR} and \cite{YOU}.

The notion of removability is closely related to Hausdorff measure. Indeed, a result generally attributed to Painlev\'e states that $E$ is removable whenever its one-dimensional Hausdorff measure $\mathcal{H}^1(E)$ is zero, see e.g. \cite[Theorem 2.7]{YOU2}. In particular, if the Hausdorff dimension of $E$ satisfies $\operatorname{dim}_H(E)<1$, then $\gamma(E)=0$. On the other hand, a simple argument using Cauchy transforms and Frostman's Lemma can be used to deduce that $\gamma(E)>0$ whenever $\operatorname{dim}_H(E)>1$, see e.g. \cite[Theorem 2.10]{YOU2}. Painlev\'e's problem is therefore reduced to the case of dimension precisely equal to one. Over the years, however, it became apparent that this remaining case is much more difficult, and it took more than a hundred years until a reasonable solution to Painlev\'e's problem was obtained, thanks to the work of David \cite{DAV}, Tolsa \cite{TOL2} and many others.

Another mysterious property of analytic capacity is its behavior under various forms of planar transformations. For instance, until the recent groundbreaking work of Tolsa \cite{TOL2} it was not known whether sets of analytic capacity zero are preserved by simple affine maps such as $\phi(x,y):=(x,2y)$, $x,y \in \mR$. In the subsequent seminal article \cite{TOL4}, Tolsa studied the behavior of analytic capacity under bilipschitz homeomorphisms of the plane and showed that sets of analytic capacity zero are bilipschitz-invariant.

In contrast to Tolsa's result, Ransford, Younsi and Ai recently proved in \cite{RYA} that analytic capacity need not vary continuously under certain ``nice'' planar transformations known as holomorphic motions. A \textit{holomorphic motion} of the Riemann Sphere $\hC$ is a map $h:\mD \times \hC \to \hC$, where $\mD:=\{z \in \mC: |z|<1\}$ is the open unit disk, such that

\begin{enumerate}[\rm(i)]
\item for each fixed $z\in \hC$, the map $\lambda\mapsto h(\lambda,z)$
is holomorphic on $\mD$,
\item for each fixed $\lambda \in \mD$, the map $z\mapsto h(\lambda,z)$ is
injective on $\hC$,
\item $h(0,z)=z$ for all $z\in \hC$.
\end{enumerate}

For a holomorphic motion $h:\mD \times \hC \to \hC$ and a subset $E \subset \hC$, we write
$$h_\lambda(z):=h(\lambda,z) \qquad (\lambda \in \mD, z\in\hC)$$
and
$$E_\lambda:=h_\lambda(E),$$
so that $E_0=E$. In this article, we shall only consider holomorphic motions that fix the point $\infty$, i.e. $h_\lambda(\infty)=\infty$ for all $\lambda \in \mD$.

Holomorphic motions were introduced by Ma\~{n}\'{e}, Sad and Sullivan in the 1980's, motivated by applications to holomorphic dynamics. In \cite{MSS}, they proved that every holomorphic motion $h:\mD \times \hC \to \hC$ is jointly continuous in $(\lambda,z)$, which is part of a more general result known as the $\lambda$-lemma. For more information on holomorphic motions, we refer the reader to \cite[Chapter 12]{AIM} and \cite{AM}.

Now, consider a real-valued function $\mathcal{F}$ defined on the collection of all compact subsets of $\mC$. Let $E \subset \mC$ be compact, and suppose that $h:\mD \times \hC \to \hC$ is a holomorphic motion of $\hC$. Then, for each $\lambda \in \mD$, the set $E_\lambda \subset \mC$ is compact by the $\lambda$-lemma, so the function
\begin{equation}
\label{eqfunction}
\lambda \mapsto \mathcal{F}(E_\lambda) \qquad (\lambda \in \mD)
\end{equation}
is well-defined. The study of the regularity of the function (\ref{eqfunction}) for different choices of $\mathcal{F}$ has been a unifying theme in the theory of holomorphic motions as well as the subject of an abundance of classical and more recent work. For instance, the variation of Hausdorff dimension under holomorphic motions was studied in the seminal work of Ruelle \cite{RUE}, who proved that the function (\ref{eqfunction}) is real-analytic when $\mathcal{F}=\operatorname{dim}_H$ and the sets $E_\lambda$ are Julia sets of hyperbolic rational maps depending holomorphically on $\lambda$. This result is considered one of the landmarks of thermodynamic formalism.

In general, Hausdorff dimension may not vary real analytically under holomorphic motions. On the other hand, it always changes continuously, and the same is true for area measure. The study of the behavior of Hausdorff dimension and area under holomorphic motions played a fundamental role in the influential work of Astala \cite{AST} on quasiconformal mappings. For the latest developments on the subject, see the recent paper \cite{FRY} by Fuhrer, Ransford and Younsi.

In recent years, there has also been significant interest in the study of the behavior of various capacities under holomorphic motions. In particular, it was observed that the function
$$\lambda \mapsto \gamma(E_\lambda) \qquad (\lambda \in \mD)$$
may not behave as nicely as one would expect. In \cite{PRY}, Pouliasis, Ransford and the Younsi showed that there exist a compact set $E$ with $\gamma(E)>0$ and a holomorphic motion $h:\mD \times \hC \to \hC$ for which the functions $\lambda \mapsto \gamma(E_\lambda)$ and $\lambda \mapsto \log{\gamma(E_\lambda)}$ are neither subharmonic nor superharmonic on $\mD$. See also the work of Zakeri in \cite{ZAK} for related results and applications to holomorphic dynamics.

In fact, it turns out that the function $\lambda \mapsto \gamma(E_\lambda)$ need not be continuous either, as previously mentioned.

\begin{theorem}[Ransford--Younsi--Ai \cite{RYA}]
\label{capmotionthm}
There exist a compact set $E \subset \mC$ and a holomorphic motion $h:\mD \times \hC \to \hC$ for which the function
$$\lambda \mapsto \gamma(E_\lambda) \qquad (\lambda \in \mD)$$
is discontinuous at $0$.
\end{theorem}

In the same paper, the authors proved that logarithmic capacity, on the other hand, does vary continuously under holomorphic motions. See also \cite{POU} for the corresponding result for condenser capacity.

The present work is motivated by the following question raised by Paul Gauthier:

\begin{question}
\label{mainquestion}
Does Theorem \ref{capmotionthm} remain true if analytic capacity is replaced by continuous analytic capacity?
\end{question}

In other words, the question is whether continuous analytic capacity vary continuously under holomorphic motions. The \textit{continuous analytic capacity} of a compact set $E \subset \mC$ is defined by
$$\alpha(E):= \sup \{|g'(\infty)|: g \in A(D), |g| \leq 1 \,\, \mbox{on} \,\, \mC\}.$$
Here as before $D=\hC \setminus E$, and $A(D)$ denotes the subspace of $H^\infty(D)$ consisting of continuous functions on $\hC$ that are analytic on $D$.

The notion of continuous analytic capacity was introduced by Erokhin and Vitushkin in the 1950's to study problems of uniform rational approximation of analytic functions on compact subsets of the plane. See \cite{VIT}. It follows directly from the definitions that $\alpha(E) \leq \gamma(E)$ for all $E$. In particular, we have that $\alpha(E)=0$ whenever $\mathcal{H}^1(E)=0$, in view of previous remarks. In fact, Cauchy's theorem can be used to show that $\alpha(E)=0$ if $\mathcal{H}^1(E)<\infty$, see e.g. \cite[Theorem 3.4]{YOU2}. More generally, this remains true if $E$ has $\sigma$-finite $\mathcal{H}^1$ measure. On the other hand, the aforementioned argument using Cauchy transforms and Frostman's Lemma applies to continuous analytic capacity as well, so that $\alpha(E)>0$ whenever $\operatorname{dim}_H(E)>1$. The case of dimension one is much less understood for continuous analytic capacity than for analytic capacity. As far as we know, there is no known geometric characterization of sets of zero continuous analytic capacity.

Unfortunately the construction of Theorem \ref{capmotionthm} does not directly provide an answer to Question \ref{mainquestion}. Indeed, the set $E$ and the holomorphic motion $h:\mD \times \hC \to \hC$ in Theorem \ref{capmotionthm} were constructed so that $\gamma(E)=\gamma(E_0)>0$ but $\gamma(E_{\lambda_n})=0$ for all $n$, for some sequence $(\lambda_n)$ of small positive numbers converging to $0$. In fact, a key part of the construction was to ensure that the set $E$ together with the sets $E_{\lambda_n}$ were all contained in the unit circle, so that their analytic capacity was easier to estimate. Unfortunately, this constraint on the sets $E$ and $E_{\lambda_n}$ shows that the construction in Theorem \ref{capmotionthm} cannot work for Question \ref{mainquestion}. Indeed, since the sets $E$ and $E_{\lambda_n}$ are all contained in the unit circle, we have $\alpha(E)=0=\alpha(E_{\lambda_n})$ for all $n$, and there is no longer a discontinuity at $0$. This remains a fundamental obstacle and new ideas are required.

In this paper, we instead use an approach based on harmonic measure and Dirichlet algebras to answer Question \ref{mainquestion} in the affirmative. This approach also provides a new proof of Theorem \ref{capmotionthm}.

\begin{theorem}
\label{mainthm2}
There exist compact sets $E$ and $F$ as well as a holomorphic motion $h:\mD \times \hC \to \hC$ for which both functions
$$\lambda \mapsto \gamma(E_\lambda), \quad \lambda \mapsto \alpha(F_\lambda) \qquad (\lambda \in \mD)$$
are discontinuous at $0$.

\end{theorem}

We also show how to combine the sets $E$ and $F$ in Theorem \ref{mainthm2} in order to get a single compact set.

\begin{corollary}
\label{maincorollary}
There exist a compact set $K$ and a holomorphic motion $g:\mD \times \hC \to \hC$ for which both functions
$$\lambda \mapsto \gamma(K_\lambda), \quad \lambda \mapsto \alpha(K_\lambda) \qquad (\lambda \in \mD)$$
are discontinuous at $0$.
\end{corollary}

We now give the main ideas of the proof of Theorem \ref{mainthm2}, which is inspired by holomorphic dynamics. Denote by $\mathcal{M}_0$ the \textit{main cardioid} of the Mandelbrot set, that is, the set of all parameters $c \in \mC$ for which the polynomial $p_c(z):=z^2+c$ has an attracting fixed point. It is well-known that the quadratic Julia sets $\mathcal{J}_c$ with $c \in \mathcal{M}_0$ are quasicircle Jordan curves and that they move holomorphically. More specifically, there is a holomorphic motion $h:\mD \times \hC \to \hC$ such that for each $\lambda \in \mD$, the map $h_\lambda$ is the unique conformal map from $\hC \setminus \overline{\mD}$ onto $\Omega_\lambda^*$ normalized by $h_\lambda(z) = z + O(1/z)$ at $\infty$, where $\Omega_\lambda^*$ is the unbounded complementary component of the quasicircle Julia set $\mathcal{J}_{\lambda/4}$. Note that the factor $1/4$ is introduced since $\mathcal{M}_0$ does not contain the unit disk but does contain the disk centered at $0$ of radius $1/4$.

\begin{figure}[h]
\begin{center}
\scalebox{0.5}{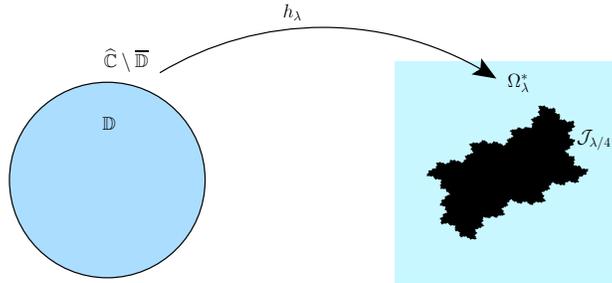}
\end{center}
\caption{The B\"{o}ttcher motion $h_\lambda(z)$.}
\end{figure}

Now, a classical theorem of Fatou states that for $\lambda \in \mD \setminus \{0\}$, the corresponding Julia set $\mathcal{J}_{\lambda/4}$ has no tangent point. As we shall see, this has two important consequences.

First, for such $\lambda$, the harmonic measure $\omega_\lambda^*$ on $\Omega_\lambda^*$ is mutually singular with $\mathcal{H}^1$ on $\mathcal{J}_{\lambda/4}$. In particular, if $(\lambda_n)$ is a fixed sequence of non-zero complex numbers in $\mD$ converging to $0$, then for each $n$ there is a Borel set $B_n \subset \mathcal{J}_{\lambda_n/4}$ with full harmonic measure but zero length. By definition of harmonic measure on Jordan curves, it follows that for each $n$, the preimage set $A_n:=h_{\lambda_n}^{-1}(B_n) \subset \partial \mD$ has full normalized Lebesgue measure in the unit circle. Letting $E$ be a compact subset of $\cap_n A_n$ of positive measure, we see that for all $n$, the set $E_{\lambda_n}=h_{\lambda_n}(E)$ has zero length, hence zero analytic capacity. On the other hand, it follows from well-known properties of analytic capacity that $\gamma(E_0)=\gamma(E)>0$, since $E$ is a subset of the unit circle with positive length. It follows that the function
$$\lambda \mapsto \gamma(E_\lambda) \qquad (\lambda \in \mD)$$
is discontinuous at $0$.

Secondly, the fact that $\mathcal{J}_{\lambda/4}$ has no tangent point for $\lambda \in \mathbb{D}\setminus \{0\}$ implies that for such $\lambda$, we have $\alpha(\mathcal{J}_{\lambda/4}) = \gamma(\mathcal{J}_{\lambda/4})$. But $\gamma(\mathcal{J}_{\lambda/4})=1$, which easily follows from the fact that $h_\lambda: \hC \setminus \overline{\mD} \to \Omega_\lambda^*$ is conformal with normalization $h_\lambda(z) = z + O(1/z)$ at $\infty$. Therefore, taking $F=\partial \mD$, we get that $\alpha(F_\lambda)=\alpha(\mathcal{J}_{\lambda/4})=1$ for all $\lambda \in \mathbb{D}\setminus \{0\}$. On the other hand, note that $\alpha(F_0)=\alpha(F)=0$, since $F=\partial \mathbb{D}$ has finite length. It follows that the function
$$\lambda \mapsto \alpha(F_\lambda) \qquad (\lambda \in \mathbb{D})$$
is discontinuous at $0$.

The proof of Theorem \ref{mainthm2} outlined above requires various properties of quasicircle Julia sets that might be of independent interest. For this reason we record them in the following proposition.

\begin{proposition}
\label{mainproposition2}
Let $c$ be a non-zero parameter in the main cardioid $\mathcal{M}_0$ of the Mandelbrot set, and let $\mathcal{J}_c$ be the corresponding Julia set. Denote by $\omega_c$ and $\omega_c^*$ the harmonic measures for the bounded and unbounded components of $\hC\setminus \mathcal{J}_c$ respectively. Then

\begin{itemize}
\item $\mathcal{J}_c$ has no tangent point,
\item $\omega_c \perp \omega_c^*$,
\item $\alpha(\mathcal{J}_c) = \gamma(\mathcal{J}_c)=1$,
\item $\omega_c^* \perp \mathcal{H}^1$ on $\mathcal{J}_c$.
\end{itemize}

\end{proposition}
As far as we know, Proposition \ref{mainproposition2} provides the first examples of fractal sets for which the precise value of the continuous analytic capacity is known (and positive).

We now present an application of our method to the study of extremal functions for continuous analytic capacity. First, it is well-known that for any compact set $E \subset \mC$, there exists an extremal function attaining the supremum in the definition of analytic capacity. More precisely, there exists a function $f$ analytic on $D=\hC \setminus E$ with $|f| \leq 1$ on $D$ and $f'(\infty)=\gamma(E)$. In fact, if $\gamma(E)>0$
 then $f$ is unique and is called the \textit{Ahlfors function} for $E$. Uniqueness was first proved by S. Ya. Khavinson in \cite{HAV2}, see also \cite{HAV} for an earlier announcement and \cite{FIS} for a simplified proof by Fisher. By contrast, we show that extremal functions for continuous analytic capacity may not exist.

\begin{corollary}
\label{extremal}
Let $J \subset \mathbb{C}$ be any Jordan arc such that $\mathcal{H}^1(T_J)=0$, where $T_J$ denotes the set of all tangent points of $J$. Then there is no continuous function $g:\hC \to \mC$ analytic on $\hC \setminus J$ such that $|g| \leq 1$ on $\mC$ and $|g'(\infty)|=\alpha(J)$
\end{corollary}

The rest of the article is structured as follows. Section \ref{sec2} contains various preliminaries and notation that will be needed throughout the paper. In Section \ref{sec3} we prove Theorem \ref{mainthm2}. Section \ref{sec4} contains the proof of Corollary \ref{extremal}. Lastly, in Section \ref{sec5} we prove Corollary \ref{maincorollary}.

\subsection*{Acknowledgments}
The author would like to thank the anonymous referee for several relevant historical remarks related to Corollary \ref{extremal}.

\section{Notation and preliminaries}
\label{sec2}

\subsection{Notation}
The following notation will be used throughout the paper. We denote the complex plane by $\mC$, the Riemann sphere (extended complex plane) by $\hC$ and the open unit disk by $\mD$. For $z_0 \in \mC$ and $r>0$, we denote by
$$B(z_0,r):=\{z \in \mC: |z-z_0|<r\}$$
and
$$\overline{B}(z_0,r):=\{z \in \mC: |z-z_0| \leq r\}$$
the open disk and closed disk respectively centered at $z_0$ of radius $r$. In particular $\mD=B(0,1)$.

Let $A$ be a subset of $\mC$. For $s \geq 0$ and $0<\delta \leq \infty$, we define
$$\mathcal{H}_{\delta}^{s}(A):= \inf \left\{ \sum_j \operatorname{diam}(A_j)^s : A \subset \bigcup_j A_j, A_j \subset \mC, \operatorname{diam}(A_j) \leq \delta \right\}.$$
Here $\operatorname{diam}(A_j)$ denotes the diameter of the set $A_j$:
$$\operatorname{diam}(A_j):=\sup_{z,w \in A_j} |z-w|.$$
The $s$-\textit{dimensional Hausdorff measure} of $A$ is
$$\mathcal{H}^s(A) := \sup_{\delta>0} \mathcal{H}_{\delta}^s(A) = \lim_{\delta \to 0} \mathcal{H}_{\delta}^s(A).$$
The \textit{Hausdorff dimension} of $A$ is the unique positive number $\operatorname{dim}_H(A)$ such that

\begin{displaymath}
\mathcal{H}^{s}(A) = \left\{ \begin{array}{ll}
\infty & \textrm{if $s < \dim_H(A) $}\\
0 & \textrm{if $s>\dim_H(A)$}.\\
\end{array} \right.
\end{displaymath}

We will also sometimes use \textit{length} to denote $1$-dimensional Hausdorff measure $\mathcal{H}^1$.

\subsection{Jordan curves and harmonic measure}

A curve $J \subset \mC$ parametrized by a continuous function $\eta:[0,1] \to \mC$ is called a \textit{Jordan arc} if $\eta$ is injective, in other words, if the curve $J$ is simple (non self-intersecting). We say that $J$ is \textit{rectifiable} if $\mathcal{H}^1(J)<\infty$.

If instead $\eta$ is injective when restricted to $(0,1]$ and if $\eta(0)=\eta(1)$, we call $J$ a \textit{Jordan curve}.

If $J$ is a Jordan arc and if $t_0 \in (0,1)$, we say that $J$ has a \textit{tangent} at $\eta(t_0)$ if there exists an angle $\theta$ such that

\begin{displaymath}
\arg{(\eta(t)-\eta(t_0))} \to \left\{ \begin{array}{ll}
\theta & \textrm{if $t \to t_0^+ $}\\
\theta+\pi & \textrm{if $t \to t_0^-$}.\\
\end{array} \right.
\end{displaymath}
This is independent of the choice of the parametric representation. The set of tangent points of $J$ will be denoted by $T_J$. Note that this definition remains valid if $J$ is a Jordan curve, and in this case we can also consider whether $\eta(0)=\eta(1)$ is a tangent point.

Let $J \subset \mC$ be a Jordan curve and denote by $\Omega$ and $\Omega^*$ the bounded and unbounded components of $\hC \setminus J$ respectively, so that $\infty \in \Omega^*$. Fix a point $z_0 \in \Omega$, and consider conformal maps $f:\mD \to \Omega$ and $g: \hC \setminus \overline{\mD} \to \Omega^*$ with $f(0)=z_0$ and $g(\infty)=\infty$. Then $f$ and $g$ extend to homeomorphisms on the closure of their respective domain by Carath\'eodory's theorem, and we can consider the pushforward measures $\omega:=f_*(\sigma), \omega^*:=g_*(\sigma)$, where $\sigma$ is the normalized Lebesgue measure on $\partial \mD$. This defines Borel probability measures on $J$ satisfying
$$\omega(E)=\sigma(f^{-1}(E))$$
and
$$\omega^*(E)=\sigma(g^{-1}(E))$$
for every Borel set $E \subset J$. The measures $\omega$ and $\omega^*$ are called \textit{harmonic measures} for $\Omega$ and $\Omega^*$ with respect to $z_0$ and $\infty$ respectively. If $J \subset \mC$ is a Jordan arc, then $\hC \setminus J$ has only one component, but since $J$ has two sides there are two measures $\omega$ and $\omega^*$ which give the harmonic measure of sets on each of the two sides of $J$.

Let $\mu_1$ and $\mu_2$ be two positive Borel measures on $\mC$. We say that $\mu_1$ is \textit{absolutely continuous} with respect to $\mu_2$, and write $\mu_1 \ll \mu_2$, if $\mu_1(E)=0$ whenever $E \subset \mC$ is a Borel set with $\mu_2(E)=0$. If $\mu_1 \ll \mu_2$ and $\mu_2 \ll \mu_1$, then we say that $\mu_1$ and $\mu_2$ are \textit{mutually absolutely continuous} and write $\mu_1 \ll \mu_2 \ll \mu_1$.

On the other hand, we say that $\mu_1$ and $\mu_2$ are \textit{mutually singular}, and write $\mu_1 \perp \mu_2$, if there exist two disjoint Borel sets $A,B \subset \mC$ such that $\mu_1$ is concentrated on $A$ and $\mu_2$ is concentrated on $B$, meaning that
$$\mu_1(E)=\mu_1(E \cap A)$$
and
$$\mu_2(E)=\mu_2(E \cap B)$$
for every Borel set $E \subset \mC$.

For example, it is well-known that harmonic measures on the same domain but for different points are always mutually absolutely continuous, see e.g. \cite[Theorem 4.3.6]{RAN2}.

\subsection{Properties of analytic capacity and continuous analytic capacity}

We now list various properties of analytic capacity and continuous analytic capacity. Recall from the introduction that for $E \subset \mC$ compact and $D:=\hC \setminus E$, the analytic capacity of $E$ is
$$\gamma(E)= \sup \{|f'(\infty)|: f \in H^\infty(D), |f| \leq 1 \,\, \mbox{on} \,\, D\}$$
and the continuous analytic capacity of $E$ is
$$\alpha(E)= \sup \{|g'(\infty)|: g \in A(D), |g| \leq 1 \,\, \mbox{on} \,\, \mC \},$$
where $H^\infty(D)$ is the space of all bounded analytic functions on $D$ and $A(D)$ is the subspace of $H^\infty(D)$ consisting of all continuous functions on $\hC$ that are analytic on $D$.

In the following, $E,F$ and $E_n$, $n \in \mathbb{N}$, all denote compact subsets of $\mC$. Also, we denote by $\Omega_E$ and $\Omega_F$ the unbounded components of $\hC \setminus E$ and $\hC \setminus F$ respectively.

\begin{enumerate}[\rm(P1)]
\item $\alpha(E) \leq \gamma(E)$.
\item For $c,d \in \mC$, $\gamma(cE+d)=|c|\gamma(E)$ and $\alpha(cE+d)=|c|\alpha(E)$.
\item If $E \subset F$, then $\gamma(E) \leq \gamma(F)$ and $\alpha(E) \leq \alpha(F)$.
\item If $E_1 \supset E_2 \supset E_3 \dots$, then $\gamma(\cap_n E_n) = \lim_{n \to \infty} \gamma(E_n)$.
\item $\gamma(E) = \gamma(\partial \Omega_E)$.
\item If $\gamma(E)>0$ and if $\Omega_E=\hC \setminus E$ (i.e. $\hC \setminus E$ is connected), then there is a unique function $f \in H^\infty(\Omega_E)$ with $|f| \leq 1$ on $\Omega_E$ and $f'(\infty)=\gamma(E)$, called the \textit{Ahlfors function} for $E$. If in addition $E$ is connected, then $f$ is the unique conformal map from $\Omega_E$ onto $\mD$ with $f(\infty)=0$ and $f'(\infty)>0$.

\item If $f:\Omega_E\to \Omega_F$ is conformal with $f(z)=az+b+O(1/z)$ at $\infty$, then $\gamma(F)=|a|\gamma(E)$.
\item For $a,b \in \mC$, $\gamma([a,b])=|a-b|/4$.
\item For $z_0 \in \mC$ and $r>0$, $\gamma(\overline{B}(z_0,r))=r$.
\item If $E$ is connected, then $\gamma(E) \geq \operatorname{diam}(E)/4$ and $\gamma(E)=c(E)$, where $c(E)$ denotes the logarithmic capacity of $E$.
\item $\gamma(E) \leq \mathcal{H}^1(E)$. In particular, if $\mathcal{H}^1(E)=0$, then $\gamma(E)=0$.
\item $\gamma(E)=0$ if and only if $E$ is removable for the class $H^\infty$, i.e., every bounded analytic function on $\hC \setminus E$ is constant.
\item $\alpha(E)=0$ if and only if $E$ is removable for the class $A$, i.e., every continuous function on $\hC$ analytic on $\hC \setminus E$ is constant.
\item If $\Omega_E=\hC \setminus E$ and if $E$ is bounded by finitely many pairwise disjoint analytic Jordan curves, then $\alpha(E)=\gamma(E)$. In particular, we have $\alpha(\overline{B}(z_0,r))=r$ for all $z_0 \in \mC, r>0$.
\item If $E$ is contained in the real line, then $\gamma(E)=\mathcal{H}^1(E)/4$.
\item If $\mathcal{H}^1(E)<\infty$, then $\gamma(E)=0$ if and only if $\mathcal{H}^1(E \cap \Gamma)=0$ for all rectifiable curves $\Gamma$. In particular, if $E$ is contained in a rectifiable curve, then $\gamma(E)=0$ if and only if $\mathcal{H}^1(E)=0$.
\item If $\mathcal{H}^1(E)<\infty$, then $\alpha(E)=0$. More generally, if $E$ has $\sigma$-finite length, then $\alpha(E)=0$.
\item If $\operatorname{dim}_H(E)>1$, then $\alpha(E)>0$ and so in particular $\gamma(E)>0$.
\item Suppose $E$ and $F$ are disjoint. If $\gamma(E)=0$, then $\gamma(E \cup F) = \gamma(F)$. If $\alpha(E)=0$, then $\alpha(E \cup F)=\alpha(F)$.
\item There is an absolute constant $C$ such that
$$\gamma(E \cup F) \leq C(\gamma(E)+\gamma(F))$$
and
$$\alpha(E \cup F) \leq C(\alpha(E)+\alpha(F)).$$
\item If $\phi: \mC \to \mC$ is bilipschitz, there is a constant $C$ depending only on $\phi$ such that
$$\frac{1}{C} \gamma(E) \leq \gamma(\phi(E)) \leq C \gamma(E)$$
and
$$\frac{1}{C} \alpha(E) \leq \alpha(\phi(E)) \leq C \alpha(E).$$
\end{enumerate}
Properties (P1) to (P13) are standard and the proofs can be found in various textbooks such as \cite{DUD}, \cite{GAR} and \cite{TOL}. Property (P14) is due to Ahlfors and Garnett, see \cite[Chapter I, Theorem 4.1]{GAR}. Property (P15) is due to Pommerenke, see \cite[Chapter I, Theorem 6.2]{GAR}. Property (P16) is a theorem due to David \cite{DAV}, formerly known as Vitushkin's conjecture. Property (P17) is essentially due to Besicovitch, see e.g. \cite[Theorem 3.4]{YOU2} and \cite[Corollary 3.5]{YOU2}. For Property (P18), see e.g. \cite[Theorem 2.1]{YOU2}. Property (P19) follows directly from \cite[Proposition 2.2]{YOU2} and \cite[Proposition 3.1]{YOU2}. Properties (P20) and (P21) are deep results due to Tolsa, see \cite{TOL2}, \cite{TOL3} and \cite{TOL4}. It is not known if we can take $C=1$ in Property (P20).

\subsection{Tangent points, harmonic measure and Dirichlet algebras}

The proof of Theorem \ref{mainthm2} relies on the existence of Jordan curves $J$ for which $\gamma(J)=\alpha(J)$. In general analytic capacity and continuous analytic capacity are equal provided $A(\hC \setminus J)$ is dense in $H^{\infty}(\hC \setminus J)$ in some sense, which is made precise using the notion of Dirichlet algebra.

\begin{definition}
For $E \subset \mC$ compact and $D=\hC \setminus E$, we say that $A(D)$ is a \textit{Dirichlet algebra} if for every continuous function $g:E \to \mR$ and every $\epsilon>0$, there exists a function $f \in A(D)$ such that
$$\|g-\operatorname{Re}(f)\|_E<\epsilon,$$
where $\|\cdot \|_E$ denotes the uniform norm on $E$.
\end{definition}

In 1963, Browder and Wermer characterized Dirichlet algebras on Jordan curves in terms of harmonic measure. For a Jordan curve $J \subset \mC$, as before we denote by $\omega$ and $\omega^*$ the harmonic measures on $\Omega$ and $\Omega^*$, the bounded and unbounded components of $\hC \setminus J$ respectively. For simplicity, let $A_J:=A(\hC \setminus J)$. Then $A_J$ is the space of all continuous functions on $\hC$ that are analytic on both $\Omega$ and $\Omega^*$.

\begin{theorem}[Browder--Wermer \cite{BW}]
\label{BrowderWermer}
The space $A_J$ is a Dirichlet algebra if and only if $\omega \perp \omega^*$.
\end{theorem}

Much later, in 1989, Bishop, Carleson, Garnett and Jones obtained a geometric characterization of the Jordan curves for which $\omega \perp \omega^*$.

\begin{theorem}[Bishop--Carleson--Garnett--Jones \cite{BCGJ}]
\label{BCGJ}
For a Jordan curve $J \subset \mC$ and $\omega,\omega^*$ as above, we have that $\omega \perp \omega^*$ if and only if $\mathcal{H}^1(T_J)=0$, where $T_J$ denotes the set of all tangent points of $J$.
\end{theorem}

The notion of Dirichlet algebra turns out to be closely related to other types of approximation.

\begin{definition}
Let $D:=\hC \setminus J$. We say that $A_J$ is \textit{pointwise boundedly dense} in $H^\infty(D)$ if there is a constant $C>0$ depending only on $J$ such that for every $f \in H^\infty(D)$, there is a sequence $(f_n) \subset A_J$ such that $\|f_n\|_D \leq C \|f\|_D$ and $f_n(z) \to f(z)$ as $n \to \infty$ for all $z \in D$.
\end{definition}

\begin{definition}
We say that $A_J$ is \textit{strongly pointwise boundedly dense} in $H^\infty(D)$ if $A_J$ is pointwise boundedly dense in $H^\infty(D)$ with $C=1$.
\end{definition}
It is quite remarkable that the three types of approximation are in fact equivalent.

\begin{theorem}
\label{thmDir}
For $J \subset \mC$ a Jordan curve and $D=\hC \setminus J$, the following are equivalent:
\begin{enumerate}[\rm(i)]
\item $A_J$ is a Dirichlet algebra.
\item $A_J$ is pointwise boundedly dense in $H^\infty(D)$.
\item $A_J$ is strongly pointwise boundedly dense in $H^\infty(D)$.
\end{enumerate}
\end{theorem}

The implication (i) $\Rightarrow$ (ii) is due to Hoffman \cite{WER} (see also \cite[Theorem 9.1]{GAG}), the implication (ii) $\Rightarrow$ (iii) is due to Davie \cite[Theorem 1.3]{DAVI} (see also \cite[Theorem 6.8]{GAG}) and the implication (iii) $\Rightarrow$ (i) is due to Gamelin and Garnett \cite[Theorem 9.1]{GAG}. See also \cite{BIS} and \cite{BIS2} for alternate proofs.

\subsection{B\"{o}tcher holomorphic motion}

As mentioned in the introduction, the holomorphic motion $h:\mD \times \hC \to \hC$ in Theorem \ref{mainthm2} stems from the theory of holomorphic dynamics. We need some preliminaries.

For a quadratic polynomial $p_c(z):=z^2+c$, we define the \textit{filled Julia set} $\mathcal{K}_c$ of $p_c$ by
$$\mathcal{K}_c:=\{z \in \mC : p_c^m(z) \nrightarrow \infty \,\, \mbox{as}\,\, m \to \infty\},$$
where $p_c^m$ denotes the $m$-th iterate of the polynomial $p_c$. The \textit{Julia set} $\mathcal{J}_c$ of $p_c$ is defined as the boundary of the filled Julia set:
$$\mathcal{J}_c:=\partial \mathcal{K}_c,$$
and the \textit{Mandelbrot set} $\mathcal{M}$ is defined as the set of all parameters $c \in \mC$ for which the orbit of $0$ under $p_c$ remains bounded:
$$\mathcal{M}:= \{c \in \mC : p_c^m(0) \nrightarrow \infty \,\, \mbox{as}\,\, m \to \infty\}.$$
Note that a parameter $c$ belongs to $\mathcal{M}$ if and only if the corresponding Julia set $\mathcal{J}_c$ is connected. For $c \notin \mathcal{M}$, the Julia set $\mathcal{J}_c$ is a Cantor set. The Mandelbrot set contains a main cardioid $\mathcal{M}_0$ defined as the set of all parameters $c$ for which the polynomial $p_c$ has an attracting fixed point. It is easy to see that $\mathcal{M}_0$ contains $B(0,1/4)$.
For background on Julia sets and the Mandelbrot set, the reader may consult \cite{MCM} and \cite{MIL}.

The main cardioid $\mathcal{M}_0$ is especially interesting for Theorem \ref{mainthm2} since it is well-known that quadratic Julia sets with parameters belonging to $\mathcal{M}_0$ move holomorphically. More precisely, for each $c \in \mathcal{M}_0$, there is a unique conformal map $B_c:\hC \setminus \overline{\mD} \to \hC\setminus \mathcal{K}_c$ with normalization $B_c(z) = z + O(1/z)$ at $\infty$, called the \textit{B\"{o}ttcher map}. The map $B(c,z):=B_c(z)$ is holomorphic in both variables. Note that $B_0$ is the identity since $\mathcal{J}_0$ is the unit circle.

In order to work on the unit disk rather than $\mathcal{M}_0$, we make the change of variable $c=\lambda/4$. For $\lambda\in{\mathbb D}$, we denote by $\Omega_\lambda$ and $\Omega_\lambda^*$ the bounded and unbounded components of $\hC \setminus \mathcal{J}_{\lambda/4}$ respectively. Then the map $h: \mD \times (\hC \setminus \overline{\mD}) \to \hC$ defined by
$$h(\lambda,z):=B_{\lambda/4}(z) \qquad (\lambda \in \mD, z \in \hC \setminus \overline{\mD}) $$
 gives a holomorphic motion of $\hC \setminus \overline{\mD}$. By a famous theorem of S\l odkowski (see \cite{SLO} or \cite[Theorem 12.3.2]{AIM}), the holomorphic motion $h$ extends to a holomorphic motion of the whole sphere, which we denote by the same letter and refer to as the \textit{B\"{o}ttcher motion}. For each $\lambda \in \mD$, the map $h_\lambda:\hC \to \hC$ is quasiconformal. Moreover, it maps $\hC \setminus \overline{\mD}$ conformally onto $\Omega_{\lambda}^*$, and the unit circle homeomorphically onto $\mathcal{J}_{\lambda/4}$. See Figure 1 from the introduction.
In particular, this shows that the Julia sets $\mathcal{J}_{\lambda/4}$ for $\lambda \in \mD$ are all quasicircles. We will also need the following classical theorem of Fatou.

\begin{theorem}[Fatou]
\label{lemmaFatou}
For $c \in \mathcal{M}_0 \setminus \{0\}$, the corresponding Julia set $\mathcal{J}_c$ has no tangent point.
\end{theorem}
See \cite[Chapter 5, Section 3, Theorem 1]{STE} for a proof.

\section{Proof of Theorem \ref{mainthm2}}
\label{sec3}

We now have everything needed in order to prove Theorem \ref{mainthm2}. As mentioned in the introduction, the proof relies on the properties of the Julia sets in Proposition \ref{mainproposition2}, which we prove first.

\begin{proof}
Let $c$ be a non-zero parameter in the main cardioid $\mathcal{M}_0$ of the Mandelbrot set, and let $\mathcal{J}_c$ be the corresponding Julia set. Denote by $\omega_c$ and $\omega_c^*$ the harmonic measures for the bounded and unbounded components of $\hC\setminus \mathcal{J}_c$ respectively.

The fact that $\mathcal{J}_c$ has no tangent point is precisely Fatou's Theorem \ref{lemmaFatou}. The fact that $\omega_c \perp \omega_c^*$ then follows directly from the Bishop--Carleson--Garnett--Jones Theorem \ref{BCGJ}.

We now prove that $\alpha(\mathcal{J}_c) = \gamma(\mathcal{J}_c)=1$. In view of Property (P1) from Section \ref{sec2}, it suffices to show that $\alpha(\mathcal{J}_c) \geq \gamma(\mathcal{J}_c)$. To see this, first note that if $D_c:=\hC \setminus \mathcal{J}_c$, then $A(D_c)$ is a Dirichlet algebra, by the Browder--Wermer Theorem \ref{BrowderWermer}. In particular, by Theorem \ref{thmDir}, the space $A(D_c)$ is strongly pointwise boundedly dense in $H^\infty(D_c)$. Now, let $\epsilon>0$, and let $f \in H^\infty(D_c)$ with $|f| \leq 1$ on $D_c$ and $|f'(\infty)| \geq \gamma(\mathcal{J}_c) - \epsilon/2$. There is a sequence $(f_n) \subset A(D_c)$ with $|f_n| \leq 1$ on $D_c$ and $f_n \to f$ pointwise on $D_c$. In fact $f_n \to f$ locally uniformly on $D_c$, by Vitali's convergence theorem for analytic functions. In particular, for $n$ large enough, we have $|f_n'(\infty)| \geq |f'(\infty)| - \epsilon/2$, and thus
$$\alpha(\mathcal{J}_c) \geq |f_n'(\infty)| \geq |f'(\infty)| - \frac{\epsilon}{2} \geq \gamma(\mathcal{J}_c) - \epsilon.$$
Since $\epsilon>0$ was arbitrary, we obtain $\alpha(\mathcal{J}_c) \geq \gamma(\mathcal{J}_c)$, as required.

It remains to show that $\omega_c^* \perp \mathcal{H}^1$ on $\mathcal{J}_c$. To see this, first note that by the Makarov compression theorem (see \cite[Section 6.6]{POM}), there is a partition
$$\mathcal{J}_c=S_c \cup A_c \cup B_c$$
where $\mathcal{H}^1(A_c)=0$ and $\omega_c^*(B_c)=0$. It is well-known that for quasicircles, the set $S_c$ coincides with the set of tangent points up to a set of zero length, see e.g. \cite[Proposition 6.28]{POM}. But $\mathcal{J}_c$ has no tangent point, thus $\mathcal{H}^1(S_c)=0$ and the measure $\mathcal{H}^1$ restricted to $\mathcal{J}_c$ is concentrated on $B_c$, a set of zero harmonic measure. In other words  $\omega_c^* \perp \mathcal{H}^1$ as required. This completes the proof of Proposition \ref{mainproposition2}.
\end{proof}

\begin{remark}
The fact that $\omega_c^* \perp \mathcal{H}^1$ on $\mathcal{J}_c$ also follows from a deep theorem of Zdunik, as mentioned to us by Saeed Zakeri. In \cite[Theorem 1]{ZDU}, Zdunik proved that if $f:\hC \to \hC$ is a rational map of degree at least $2$, then the measure of maximal entropy $\mu$ on $\mathcal{J}_f$ is mutually singular with $\mathcal{H}^\alpha$, where $\alpha$ is the dimension of $\mu$, except for the case when $f$ is critically finite with parabolic orbifold. For quadratic polynomials $z \mapsto z^2+c$ with $c \in \mathcal{M}_0$, the measure of maximal entropy $\mu$ coincides with harmonic measure at $\infty$, and $\alpha=1$ by a famous theorem of Makarov. Note that the case $c=0$ is the only quadratic polynomial with parabolic orbifold.
\end{remark}

We can now proceed with the proof of Theorem \ref{mainthm2}.

\begin{proof}
We have to construct two compact sets $E$ and $F$ as well as a holomorphic motion $h:\mD \times \hC \to \hC$ for which both functions
$$\lambda \mapsto \gamma(E_\lambda), \quad \lambda \mapsto \alpha(F_\lambda) \qquad (\lambda \in \mD)$$
are discontinuous at $0$.

Let $h:\mD \times \hC \to \hC$ be the B\"{o}tcher holomorphic motion. Recall that for each $\lambda \in \mD$, the map $h_\lambda: \hC \to \hC$ is quasiconformal. Moreover, it maps the unit circle homeomorphically onto the quasicircle $\mathcal{J}_{\lambda/4}$, and $\hC \setminus \overline{\mD}$ conformally onto $\Omega_{\lambda}^*$, the unbounded complementary component of $\mathcal{J}_{\lambda/4}$. In addition, each $h_\lambda$ has normalization $h_\lambda(z)=z+O(1/z)$ at $\infty$.

Now, by Proposition \ref{mainproposition2}, we have that for each $\lambda \in \mD \setminus \{0\}$, the harmonic measure $\omega_\lambda^*$ on $\Omega_\lambda^*$ is mutually singular with $\mathcal{H}^1$ on $\mathcal{J}_{\lambda/4}$. Let $(\lambda_n)$ be any sequence of non-zero complex numbers in $\mD$ converging to $0$. Then for each $n$, there is a Borel set $B_n \subset \mathcal{J}_{\lambda_n/4}$ with $\omega_{\lambda_n}^*(B_n)=1$ but $\mathcal{H}^1(B_n)=0$. By the definition of harmonic measure in Section \ref{sec2}, it follows that for each $n$, the preimage set $A_n:=h_{\lambda_n}^{-1}(B_n) \subset \partial \mD$ has full normalized Lebesgue measure in the unit circle. In particular, the intersection $\cap_n A_n$ also has full Lebesgue measure. Let $E$ be any compact subset of $\cap_n A_n$ with positive Lebesgue measure. Then for each $n$ we have
$$\mathcal{H}^1(E_{\lambda_n}) = \mathcal{H}^1(h_{\lambda_n}(E)) \leq \mathcal{H}^1(h_{\lambda_n}(A_n)) = \mathcal{H}^1(B_n)=0.$$
In particular, by Property (P11) in Section \ref{sec2}, we get that $\gamma(E_{\lambda_n})=0$ for all $n$. On the other hand, it follows from Property (P16) in Section \ref{sec2} that $\gamma(E_0) = \gamma(E)>0$, since $E \subset \partial \mD$ has positive Lebesgue measure. This shows that the function
$$\lambda \mapsto \gamma(E_\lambda) \qquad (\lambda \in \mD)$$
is discontinuous at $0$.

For continuous analytic capacity, note that $\alpha(\mathcal{J}_{\lambda/4}) = \gamma(\mathcal{J}_{\lambda/4})$ for all $\lambda \in \mD \setminus \{0\}$, again by Proposition \ref{mainproposition2}. But $\gamma(\mathcal{J}_{\lambda/4})= \gamma(\overline{\mD})=1$, by Property (P7) and Property (P9) in Section \ref{sec2}. Letting $F:=\partial \mD$, we obtain
$$\alpha(F_\lambda)=\alpha(\mathcal{J}_{\lambda/4}) = \gamma(\mathcal{J}_{\lambda/4}) =1,$$
for all $\lambda \in \mD\setminus \{0\}$. On the other hand, it follows from Property (P17) in Section \ref{sec2} that $\alpha(F_0)=\alpha(F)=0$, since $F$ has finite length. Thus the function
$$\lambda \mapsto \alpha(F_\lambda) \qquad (\lambda \in \mD)$$
is discontinuous at $0$, as required.
\end{proof}

\section{Proof of Corollary \ref{extremal}}
\label{sec4}

In this section, we prove Corollary \ref{extremal}.

\begin{proof}
Let $J \subset \mC$ be a Jordan arc such that $\mathcal{H}^1(T_J)=0$, where $T_J$ denotes the set of all tangent points of $J$. We have to show that there is no continuous function $g:\hC \to \mC$ analytic on $\hC \setminus J$ such that $|g| \leq 1$ on $\mC$ and $|g'(\infty)|=\alpha(J)$.

Suppose for a contradiction that an extremal function $g$ exists. Multiplying by a constant of modulus $1$ if necessary, we may assume that $g'(\infty)=\alpha(J)$. Now, as in the proof of Proposition \ref{mainproposition2}, the fact that $\mathcal{H}^1(T_J)=0$ implies that $\alpha(J)=\gamma(J)$. Here we are using the fact that the Browder--Wermer Theorem \ref{BrowderWermer} and the Bishop--Carleson--Garnett--Jones Theorem \ref{BCGJ} are also valid for Jordan arcs, not just for Jordan curves. It follows that $g$ is also extremal for $\gamma(J)$. By Property (P6) from Section \ref{sec2}, we get that $g:\hC \setminus J \to \mD$ is conformal, with $g(\infty)=0$ and $g'(\infty)>0$. But since $J$ is a Jordan arc, every conformal map from $\hC \setminus J$ onto $\mD$ must be discontinuous at each point of $J$ which is not an endpoint, by the classical theory of boundary behavior of conformal maps. This contradicts the fact that $g$ is continuous everywhere. It follows that there is no extremal function for $\alpha(J)$, as required.
\end{proof}

\section{Proof of Corollary \ref{maincorollary}}
\label{sec5}

In this section, we prove Corollary \ref{maincorollary}. For the proof, we need the following result due to Pommerenke.
\begin{theorem}[Pommerenke (1960)]
\label{thmpom}
Let $F_1, \dots, F_n$ be compact subsets of $\mC$. Then for every $\epsilon>0$, there exists $\delta>0$ with the following property:

If $E_k:=F_k + d_k$ for some $d_k \in \mC$, $k=1,\dots,n$, and if all distances between $E_i$ and $E_j$ $(i \neq j)$ are greater than $\delta$, then
$$\left| \gamma\left( \bigcup_{k=1}^n E_k \right) - \sum_{k=1}^n \gamma(E_k) \right|< \epsilon.$$

\end{theorem}
See \cite{POM0} or \cite[Chapter 6, Section 5]{KU}.

We can now proceed with the proof of Corollary \ref{maincorollary}.

\begin{proof}
Let $E,F$ and $h:\mD \times \hC \to \hC$ be as in Theorem \ref{mainthm2}, so that $h$ is a holomorphic motion of $\hC$ and both functions
$$\lambda \mapsto \gamma(E_\lambda), \quad \lambda \mapsto \alpha(F_\lambda) \qquad (\lambda \in \mD)$$
are discontinuous at $0$, where $E_\lambda=h_\lambda(E)$ and $F_\lambda = h_\lambda(F)$. By construction, the set $E$ is a compact subset of $F=\partial \mD$ with $\gamma(E)>0$, and there is a sequence $(\lambda_n)$ in $\mD \setminus \{0\}$ converging to $0$ such that
$$\gamma(E_{\lambda_n})=0 \qquad (n \in \mathbb{N})$$
and
$$\alpha(F_{\lambda_n})=1=\gamma(F_{\lambda_n}) \qquad (n \in \mathbb{N}).$$

Now, take $d>0$ large enough so that for each $\lambda \in \mD$, the translate $E_\lambda + d$ is disjoint from $F_\lambda$. We shall fix the value of $d$ later. Let $K:=(E+d) \cup F$, and define $g:\mD \times K \to \hC$ by
\begin{displaymath}
g(\lambda,z) = \left\{ \begin{array}{ll}
h(\lambda,z-d)+d & \textrm{if $z \in E+d$}\\
h(\lambda,z) & \textrm{if $z \in F$}.\\
\end{array} \right.
\end{displaymath}
It is easy to see that $g$ is a holomorphic motion of $K$. By S\l odkowski's theorem, we can extend $g$ to a holomorphic motion of $\hC$, which we still denote by the same letter. Note that for $\lambda \in \mD$, we have
$$K_\lambda = g_\lambda((E+d) \cup F) = g_\lambda(E+d) \cup g_\lambda(F) = (E_\lambda + d) \cup F_\lambda.$$

We first show that the function
$$\lambda \mapsto \alpha(K_\lambda) \qquad (\lambda \in \mD)$$
is discontinuous at $0$. For this, note that by Property (P1) and Property (P2) from Section \ref{sec2}, we have
$$\alpha(E_{\lambda_n}+d)=\alpha(E_{\lambda_n}) \leq \gamma(E_{\lambda_n})=0 \qquad (n \in \mathbb{N}).$$
It then follows from Property (P19) from Section \ref{sec2} that
$$\alpha(K_{\lambda_n}) = \alpha((E_{\lambda_n}+d) \cup F_{\lambda_n})= \alpha(F_{\lambda_n})=1 \qquad (n \in \mathbb{N}).$$
On the other hand, we have that $\mathcal{H}^1(K)=\mathcal{H}^1((E+d) \cup F)<\infty$, hence $\alpha(K)=0$ by Property (P17) from Section \ref{sec2}. This shows that the function
$$\lambda \mapsto \alpha(K_\lambda) \qquad (\lambda \in \mD)$$
is discontinuous at $0$.

It remains to prove that
$$\lambda \mapsto \gamma(K_\lambda) \qquad (\lambda \in \mD)$$
is discontinuous at $0$. For this, first observe that similarly as above, we have $\gamma(K_{\lambda_n})=1$ for all $n \in \mathbb{N}$. Moreover, note that all the previously obtained estimates for $\gamma(K_{\lambda_n})$, $\alpha(K_{\lambda_n})$ and $\alpha(K)$ are independent of $d$ as long as it is large enough. We now choose the value of $d$. By Theorem \ref{thmpom} with $n=2$, $\epsilon:=\gamma(E)>0$, $F_1:=E$, $F_2:=F$, $d_1:=d$, $d_2:=0$, we can take $d>0$ large enough so that
$$| \gamma((E+d) \cup F) - \gamma(E+d)-\gamma(F) |< \gamma(E).$$
Then
\begin{eqnarray*}
\gamma((E+d) \cup F) &>& \gamma(E+d)+\gamma(F)-\gamma(E)\\
&=& \gamma(E)+\gamma(F)-\gamma(E)\\
&=& \gamma(F)=1,
\end{eqnarray*}
where we used Property (P2), Property (P5) and Property (P9) from Section \ref{sec2}. Summarizing, we have $\gamma(K_{\lambda_n})=1$ for all $n \in \mathbb{N}$ but $\gamma(K)=\gamma((E+d) \cup F)>1$. It follows that the function
$$\lambda \mapsto \gamma(K_\lambda) \qquad (\lambda \in \mD)$$
is discontinuous at $0$, as required.
\end{proof}

\bibliographystyle{amsplain}

\end{document}